\documentclass[12pt]{amsart}

\usepackage{amsmath}

\usepackage{amsfonts,amssymb,amsthm,enumitem}
\usepackage[margin=1.2in]{geometry}
\usepackage{mathtools}
\usepackage{nicefrac,setspace}
\usepackage{xcolor}
\usepackage{hyperref}       
\usepackage{url}            
\usepackage{booktabs}       
\usepackage{amsfonts}       
\usepackage{nicefrac}       
\usepackage{microtype}      

\makeatletter
\newcommand*\bigcdot{\mathpalette\bigcdot@{.5}}
\newcommand*\bigcdot@[2]{\mathbin{\vcenter{\hbox{\scalebox{#2}{$\m@th#1\odot$}}}}}
\makeatother

\newtheorem{theorem}{Theorem}

\newtheorem*{theorem*}{Theorem}
\newtheorem*{remark}{Remark}
\newtheorem*{claim*}{Claim}
\newtheorem*{remark*}{Remark}
\newtheorem*{lemma*}{Lemma}

\newtheorem{lemma}[theorem]{Lemma}
\newtheorem{fact}[theorem]{Fact}

\newtheorem{prop}[theorem]{Proposition}

\renewcommand{\mod}{\mathop {\mathsf{mod}}}

\newcommand{\R}{\mathbb{R}}

\newcommand{\lin}[1]{#1_\mathsf{lin}}

\newcommand{\Expect}[1]{\mathop{\mathbb{E}}\left
	[ #1 \right ]}
\newcommand{\Ex}[2]{\mathop{\mathbb{E}}\displaylimits_{#1}\left
	[ #2 \right ]}

\title{An Elementary Exposition of Pisier's inequality}

\begin{document}

\author[Iyer et al.]{Siddharth Iyer}
\address{School of Computer Science, University of Washington}
\email{sviyer97@gmail.com}
\author[]{Anup Rao}
\address{School of Computer Science, University of Washington}
\email{anuprao@cs.washington.edu}
\author[]{Victor Reis}
\address{School of Computer Science, University of Washington}
\email{voreis@cs.washington.edu}
\author[]{Thomas Rothvoss}
\address{School of Computer Science, University of Washington}
\email{rothvoss@uw.edu}
\author[]{Amir Yehudayoff}
\address{Department of Mathematics, Technion-IIT}
\email{amir.yehudayoff@gmail.com}

\begin{abstract}
Pisier's inequality is central in the study of normed
spaces and has important applications in geometry. 
We provide an elementary proof of this inequality,
which avoids some non-constructive steps from previous proofs.
Our goal is to make the inequality and its proof more accessible,
because we think they will find additional applications.
We demonstrate this with a new type of restriction
on the Fourier spectrum of bounded functions on the discrete cube.
\end{abstract}

\maketitle

\section{Introduction}



The \emph{Rademacher projection} is a method to linearize
functions from the discrete cube $\{\pm 1\}^n$ to the Euclidean space $\R^m$.
It is fundamental in the study of normed spaces~\cite{maurey1976series,artstein2015asymptotic}.
Pisier's inequality controls the operator norm of the Rademacher
projection~\cite{pisier1979espaces,pisier1980theoreme}.

This inequality has several important geometric applications. 
Most strikingly, if combined with a result of Figiel and Tomczak-Jaegermann~\cite{figiel1979projections} it implies the $MM^*$-estimate,
which says that in a certain average sense, symmetric convex bodies behave much more like ellipsoids than one could derive from John's classical theorem~\cite{john1948extremum}.
The $MM^*$-estimate is, in turn, a central piece in the proof of
Milman's QS-theorem~\cite{milman1985almost,milman1986inegalite,milman1988isomorphic}, which is one of the deepest results in convex geometry.

Pisier's original proof uses complex analysis and interpolation
(and provides additional information).
Bourgain and Milman found a different and more direct proof~\cite{bourgain1987new}.
Their proof relies on several deep results,
like the Hahn-Banach theorem, the Riesz representation theorem, and Bernstein's theorem from approximation theory.

The purpose of this note is to present an elementary and accessible
proof of Pisier's inequality.
Our proof is explicit and avoids the non-constructive 
part in the proof from~\cite{bourgain1987new}.

\subsection{The inequality}
The Rademacher projection is based on Fourier analysis.
The starting point is the space of functions from $\{\pm 1\}^n$ to $\R$.
The characters form an important (orthonormal) basis for this space.
The character that corresponds to the set $S \subseteq [n]$
is the map $\chi_S : \{\pm 1\}^n \to \R$ defined by
$$\chi_S(x) = \chi_S(x_1,x_2,\ldots,x_n) = \prod_{j \in S} x_j.$$
Every $f: \{\pm 1\}^n \to \R^m$ can be uniquely expressed as 
$$f(x) = \sum_{S \subseteq [n]} \hat{f}(S) \cdot \chi_S(x)$$ where the vectors $\hat{f}(S) \in \R^m$ are the Fourier coefficients of $f$. 
The linear part of $f$ is
$$\lin{f}(x) =\sum_{S \subseteq [n]: |S|=1} \hat{f}(S) \cdot \chi_S(x)
= \sum_{j=1}^n \hat{f}(\{j\}) \cdot x_j .$$
The Rademacher projection is the map $f \mapsto \lin{f}$.
Pisier's inequality gives an upper bound on its operator norm.

\begin{theorem*}[Pisier]
\label{thm:P}
There is a constant $C>0$ so that the following holds.
Let $\| \cdot \|$ be a norm on $\R^m$.
Let $X$ be uniformly distributed in $ \{\pm 1\}^n$.
Then
$$ \Expect{\|\lin{f}(X)\|^2}^{1/2} \leq C \log(m+1) \cdot \Expect{\|f(X)\|^2}^{1/2}.$$ 
\end{theorem*}

%
%
%
The proof of Pisier's inequality from~\cite{bourgain1987new}
is based on the existence of a function $g: \{\pm 1\}^n \rightarrow \R$ that is nearly linear, yet has small $\ell_1$ norm. 
The existence of $g$ is proved in a non-constructive way.
We give an explicit and simple formula for such a function~$g$.

When $\| \cdot\|$ is the Euclidean norm,
the $C \log(m)$ term can be replaced by $1$,
because orthogonal projections do not increase the Euclidean norm.
Bourgain, however, showed that for general norms the $\log(m)$ factor is necessary~\cite{bourgain1984martingales}. 
Bourgain's construction is probabilistic. 
In Section~\ref{sec:B}, we describe a simple  explicit example, also based on Bourgain's idea, showing that $\frac{\log(m)}{\log \log (m)}$ factor is necessary.

There is a variant of Pisier's inequality for functions $f : \mathbb{R}^n \to \mathbb{R}^m$ where $X$
is Gaussian. While such a variant is useful for applications, it is a statement about an infinite-dimensional vector space of functions, which makes the proof more complicated. However, one can show 
that the variant on the discrete cube and the variant in Gaussian space are equivalent (see e.g.~\cite{artstein2015asymptotic}).

We conclude the introduction with one more application. 
Fourier analysis of Boolean functions is an important
area in computer science and mathematics with many applications (see the textbook~\cite{o2014analysis}).
A central goal is to identify properties of the Fourier spectrum
of Boolean or bounded function on the cube; see~\cite{dinur2006fourier,keller2012note}
and references within.
Pisier's inequality implies the following
restriction on the Fourier spectrum.
There is a constant $c>0$ so that
for every $f : \{\pm 1\}^n \to [-1,1]$,
$$\log(\|\hat{f}\|_0) \geq c \sum_{j=1}^n |\hat{f}(\{j\})|,$$
where $\|\hat{f}\|_0$ is the sparsity of $\hat{f}$;
i.e., the number of sets $S \subseteq [n]$
so that $\hat{f}(S) \neq 0$.
The proof of this inequality
and its sharpness can be deduced from Section~\ref{sec:B}.

\section{Preliminaries}

Convolution is a powerful tool when there is an underlying group structure.
Here the group is the cube $\{\pm 1\}^n$
with the operation $x \odot z = (x_1 z_1, \dotsc , x_n z_n)$. 
The convolution of a (vector-valued) 
function $f: \{\pm 1\}^n \rightarrow \R^m$ and 
a (scalar-valued) function $g:\{\pm 1\}^n \rightarrow \R$ is 
the function $f*g: \{\pm 1\}^n \to \R^m$ defined by
\begin{align*}
 f*g(x) = \Ex{Z}{g(Z) \cdot f(x \odot Z)}
\end{align*}
where $Z$ is uniformly random in $\{\pm 1\}^n$. 
We list some basic properties of convolution.

\begin{fact}
\label{fact:lin}
If $T:\R^m \to \R^m$ is a linear map then $T(f*g) = T(f) * g$. 
\end{fact}

\begin{fact}
\label{fact:f*gS}
$\widehat{f*g}(S) = \hat{g}(S) \cdot \hat{f}(S) $ for every $S \subseteq [n]$.
\end{fact}
\begin{proof}
 \begin{align*}
 f*g(x) & = \Expect{g(Z) \cdot f(x \odot Z)} \\ 
 & =\Expect{ \sum_{S} \hat{g}(S) \chi_S(Z) \cdot \sum_{T} \hat{f}(T) \chi_T(x \odot Z)} \\ 
  & =\Expect{ \sum_{S} \hat{g}(S) \chi_S(Z) \cdot \sum_{T} \hat{f}(T) \chi_T(x)\chi_T(Z)} \\ 
 &  =\sum_{S} \hat{g}(S) \hat{f}(S) \chi_S( x),
 \end{align*}
 where the last equality uses linearity of expectation and
 the orthonormality of the characters:
 \begin{align*}
 	\Expect{\chi_S(Z) \chi_T(Z)} = \begin{cases}
 	1 & \text{if $S=T$,}\\
 	0 & \text{otherwise.}
 	\end{cases} 
 \end{align*}
\end{proof}

\begin{fact} \label{fact:l1} For any norm $\| \cdot \|$, 
	$$\Expect{\|f*g(X) \|^2}^{1/2} \leq \Expect{|g(X)|} \cdot \Expect{\|f(X) \|^2}^{1/2}.$$
\end{fact}
\begin{proof}
	\begin{align*}
	\Expect{\|f*g(X) \|^2} & = \Ex{X}{\|\Ex{Z} {g(Z) \cdot f(X \odot Z)} \|^2}	\\ & \leq \Ex{X}{\big(\Ex{Z}{|g(Z)| \cdot  \|f(X \odot Z)\|}\big)^2},
	\end{align*}
	where the inequality follows from the convexity of the norm $\| \cdot \|$.  By the Cauchy-Schwarz inequality, we get
	\begin{align*}
	& \leq \Ex{X}{\Ex{Z}{|g(Z)|} \cdot \Ex{Z'}{|g(Z')|\cdot  \|f(X \odot Z')\|^2}} \\
	&= \Ex{Z}{|g(Z)|} \cdot \Ex{Z'}{|g(Z')| \cdot \Ex{X}  {\|f(X)\|^2}} \\
	&= \big(\Ex{Z}{|g(Z)|}\big)^2 \cdot \Ex{X}  {\|f(X)\|^2}. \qedhere
	\end{align*}
\end{proof}
%

%

\section{An Overview of the Proof}

The linear part $\lin{f}$ of $f$ can be expressed as the convolution
of $f$ with the linear function $L = \sum_{j=1}^n x_j$;
see Fact~\ref{fact:f*gS}.
In order to analyze the norm of $\lin{f} = f * L$,
we use an auxiliary function $P$ which serves as a proxy for $L$.
We call the function $P$ the {\em linear proxy},
and it depends on a parameter $\ell$ that will be set to
be $\approx \log(m)$.

\begin{lemma}
For every odd $\ell > 0$, there is
$P : \{\pm 1\}^n \to \R$ so that the following hold.
First, $P$ is close to $L$: for all $S \subseteq [n]$,
$$|\widehat{P-L}(S)| \leq \frac{8\ell}{2^\ell}.$$ 
Second, $P$ has small $\ell_1$ norm:
$$\Expect{|P(X)|} \leq 8 \ell.$$

\end{lemma}

Let us explain how to prove Pisier's inequality using the linear proxy $P$.
The convexity of norms allows to split the bound
to two terms:
\begin{align*}
\Expect{\|\lin{f}(X)\|^2}^{1/2}  
& = \Expect{\|f*L(X)\|^2}^{1/2} \\
&= \Expect{\|f*P(X) + f*(L-P)(X)\|^2}^{1/2}\\
& \leq \Expect{\|f*P(X) \|^2}^{1/2} + \Expect{\| f*(L-P)(X)\|^2}^{1/2}.
\end{align*}
Bound each of the two terms separately. 
To bound the first term, apply Fact \ref{fact:l1} and 
use the choice of $P$,
\begin{align*}
\Expect{\|f*P(X) \|^2}^{1/2} \leq \Expect{|P(Z)|} \cdot \Expect{\|f(X) \|^2}^{1/2} 
\leq 8 \ell \Expect{\|f(X) \|^2}^{1/2}.
\end{align*}
To bound the second term, we use John's theorem, which is classical
and we do not prove here. John's theorem states that
there is an invertible linear map $T : \R^m \to \R^m$
so that for every $x \in \R^m$,
$$\|T(x)\|_2 \leq \|x\| \leq \sqrt{m} \cdot \|T(x)\|_2.$$
Using $T$ we can switch between $\|\cdot\|$
and $\|\cdot \|_2$:
\begin{align*}
\Expect{\| f*(L-P)(X)\|^2}^{1/2} 
&  \leq \sqrt{m} \cdot \Expect{\| T(f*(L-P)(X)) \|_2^2}^{1/2}  \\
&  = \sqrt{m} \cdot \Expect{\| T(f)*(L-P)(X)\|_2^2}^{1/2} \\
& =  \sqrt{m} \cdot \sqrt{\sum_S \|\widehat{T(f)}(S)\|_2^2 \cdot (\widehat{L-P}(S))^2} \\
& \leq \frac{8\ell \sqrt{m}}{2^\ell} \cdot \sqrt{\sum_S \|\widehat{T(f)}(S)\|_2^2} \\
&= \frac{8\ell \sqrt{m}}{2^\ell} \cdot \Expect{\| T(f(X))\|_2^2}^{1/2}\\
& \leq \frac{8\ell \sqrt{m}}{2^\ell} \cdot \Expect{\| f(X)\|^2}^{1/2}.
\end{align*}
Putting it together,
\begin{align*}
\Expect{\|\lin{f}(X)\|^2}^{1/2}  
& \leq 8\ell \Big(1+ \frac{\sqrt{m}}{2^\ell} \Big) \Expect{\|f(X) \|^2}^{1/2}.
\end{align*}
Setting $\ell$ to be the smallest odd that
is larger than $\tfrac{1}{2} \log(m)$,
the proof is complete.

\begin{remark}
Pisier's inequality is more general than stated in Theorem~\ref{thm:P}.
The Banach-Mazur distance of the norm $\| \cdot \|$
from the Euclidean norm $\|\cdot\|_2$ is
$$D = \inf \{d \in \R: \exists T \in \mathsf{GL}_m \ \forall x \in \R^m \ \|T(x)\|_2 \leq \|x\| \leq d\cdot \|T(x)\|_2\},$$
where $\mathsf{GL}_m$ is the group of invertible linear transformations from $\R^m$ to itself. 
John's theorem states that always $D \leq \sqrt{m}$.
The above argument proves that, more generally, we can replace
the $C \log(m+1)$ term by $C \log(D+1)$.
\end{remark}

\section{Constructing the linear proxy}

The structure of the linear proxy $P$ we construct is similar
to the linear proxy from~\cite{bourgain1987new}.
However, the existence of the linear proxy in~\cite{bourgain1987new}
is proved in a non-constructive way.
Here we provide a simple and explicit formula for $P$. 
The main piece in the construction is 
the following proposition.

\begin{prop}
Let $\ell >0$ be odd and
let
$$\phi(\theta) = \frac{2\ell-1}{ \ell } \cdot \frac{\sin(\ell \theta)}{\sin^2(\theta)}.$$
There is a finitely supported distribution on $\theta \in [0,2\pi]$ such that 
\begin{align*}
\Expect{\phi(\theta) \cdot \sin^k(\theta)} = \begin{cases}
1 & \text{if $k = 1$,}\\
0 & \text{if $k = 0,2,3,\dotsc,\ell$,}
\end{cases}
\end{align*}
and
$$\Expect{|\phi(\theta)|} \leq 4\ell$$
\end{prop}

Using the proposition, the linear proxy is defined as
$$ P(x) = 2 \cdot \Ex{\theta}{\phi(\theta) \cdot \prod_{j=1}^n \Big (1+ \frac{\sin(\theta) \cdot x_j }{2} \Big)}.$$
The properties of $P$ readily follow.
To prove that $P$ is close to linear, open the product
and use linearity of expectation:
\begin{align*}
P(x) 
& =  \sum_{S \subseteq [n]}  2 \Ex{\theta}{\phi(\theta) 
\frac{\sin^{|S|}(\theta)}{2^{|S|}} } \cdot \chi_S(x).
\end{align*}
This is the Fourier representation of $P$.
The first property of $\phi$ implies that $\hat{P}(S)=0$ when $|S|=0,2,3,\dotsc,\ell$, and $\hat{P}(S)=1$ when $|S|=1$. 
When $|S|>\ell$, the second property of $\phi$ implies
\begin{align*}
|\hat{P}(S)|  \leq \frac{2}{2^{|S|}} \cdot \Expect{|\phi(\theta)|} \leq \frac{8\ell}{2^{\ell}}.
\end{align*}
Bound the $\ell_1$ norm of $P$ by
\begin{align*}
\Expect{|P(X)|} 
 &\leq 2\cdot \Expect{ | \phi(\theta)| \cdot \Big| \prod_{j=1}^n \Big (1+ \frac{\sin(\theta) \cdot X_j }{2} \Big) \Big |} \\
 &= 2 \cdot \Expect{ | \phi(\theta)| \cdot \prod_{j=1}^n \Big (1+ \frac{\sin(\theta) \cdot X_j }{2} \Big)} \\ 
& = 2\cdot \Expect{ | \phi(\theta)|} \leq 8 \ell,
\end{align*}
because $1+ \tfrac{\sin(\theta) \cdot X_j }{2} \geq 0$, and 
$\Expect{X_j} = 0$.

\subsection{Construction of $\phi$}
The cancellations below are based on the following simple fact.
Let $\Gamma$ denote the $4\ell$ equally spaced angles: 
$$\Gamma = \Big\{0,\frac{2\pi}{4\ell}, \frac{2 \cdot 2\pi}{4\ell}, \dotsc,  \frac{(4\ell-1) \cdot 2\pi}{4\ell}\Big\}.$$ 
For any integer $a$, since $\sum_{\theta \in \Gamma} e^{i a \theta} =  e^{i a \cdot \frac{ 2\pi}{4\ell}} \cdot \sum_{\theta \in \Gamma} e^{ia \theta}$,
we have 
\begin{align} \label{eqn:fourier}
\sum_{\theta \in \Gamma} e^{i a \theta} =  \begin{cases}4\ell & \text{if $a = 0 \mod 4\ell$,}\\
0 & \text{otherwise.}
\end{cases}
\end{align}

The distribution on $\theta$ is uniform in the set $\Gamma \setminus \{0, \pi\}$.
It remains to prove the stated properties of $\phi$ one-by-one.
For $k=0$,
since $\phi(\theta) = - \phi(2 \pi - \theta)$, 
$$\Expect{\phi(\theta) \sin^0(\theta)}=0.$$ For $k=1$, use the identity $\sin(\theta) = e^{-i\theta} \cdot \frac{e^{2i\theta} - 1}{2i}$:
\begin{align*}
\Expect{\phi(\theta) \cdot \sin(\theta)} & = \frac{1}{4\ell-2} \cdot \frac{2\ell-1}{\ell} \cdot \sum_{\theta \in \Gamma\setminus \{0,\pi\}} \frac{\sin(\ell \theta)}{\sin(\theta)}\\
& = \frac{1}{2\ell} \cdot \sum_{\theta \in \Gamma\setminus \{0,\pi\}} e^{-i (\ell-1) \theta} \cdot \frac{e^{i 2\ell \theta} - 1}{e^{i2 \theta} -1}\\
& = \frac{1}{2\ell} \cdot \sum_{\theta \in \Gamma\setminus \{0,\pi\}} e^{-i (\ell-1) \theta} + e^{-i(\ell -3) \theta} + \dotsc + e^{i(\ell-1)\theta}.
\end{align*}
Because $\ell$ is odd, when $\theta \in \{0,\pi\}$, we have $e^{-i(\ell-1)\theta} + \dotsc + e^{i (\ell-1)\theta}=\ell$. 
So, using \eqref{eqn:fourier}, we get
\begin{align*}
 = \frac{1}{2\ell}\cdot \Big(-2\ell+\sum_{\theta \in \Gamma} e^{-i (\ell-1) \theta} + e^{-i(\ell -3) \theta} + \dotsc + e^{i(\ell-1)\theta} \Big) = \frac{1}{2\ell} \cdot (-2\ell +4\ell)= 1.
\end{align*}
When $1 <  k \leq \ell$, because $\sin(0) = \sin(\pi) = 0$, we have
\begin{align*}
\Expect{\phi(\theta) \cdot \sin^k(\theta)} & =\frac{1}{4\ell-2} \cdot \frac{2\ell-1}{\ell} \cdot \sum_{\theta \in \Gamma\setminus \{0,\pi\}} \sin(\ell \theta) \cdot \sin^{k-2}(\theta)\\
&=\frac{1}{2\ell} \cdot  \sum_{\theta \in \Gamma} \sin(\ell \theta) \cdot \sin^{k-2}(\theta)\\
&=\frac{1}{2\ell} \cdot  \sum_{\theta \in \Gamma} \Big(\frac{e^{i\ell \theta} - e^{-i \ell \theta}}{2i} \Big)  \cdot \Big(\frac{e^{i\theta} - e^{-i \theta}}{2i} \Big)^{k-2} = 0,
\end{align*}
since every phase appearing here after opening the
parenthesis is non-zero modulo $4\ell$.

Finally, bound the $\ell_1$ norm of $\phi$: by the symmetry of $\theta$,  
\begin{align*}
\Expect{|\phi(\theta)|}   
 \leq 4 \cdot  \frac{1}{4\ell-2}  \cdot \frac{2\ell-1}{\ell}  \cdot  \sum_{j=1}^{\ell} \Big | \frac{1 }{\sin^2 (2 \pi j/(4\ell))} \Big| \leq \frac{2}{\ell}\cdot \sum_{j=1}^{\infty} \Big | \frac{\ell^2 }{j^2} \Big | \leq 4\ell,
\end{align*}
where we used the inequality $\sin(\gamma) \geq \gamma/(\pi/2)$, which is valid when $0 \leq \gamma \leq \pi/2$.

\section{A Lower Bound}
\label{sec:B}
Bourgain showed that Pisier's inequality is sharp~\cite{bourgain1984martingales}.
His example is non-explicit because it uses the probabilistic method.
Here we give a simple and explicit example showing that a loss of $\frac{\log m}{\log \log m}$ is necessary. 
The main technical ingredient is the following construction:
\begin{theorem} \label{thm:Ex}
  For any $n \in \mathbb{N}$, there is a function $F : \{ -1,1\}^n \to \mathbb{R}$ with the following properties:
 \begin{enumerate}
  \item[(A)]  $\|F\|_\infty \leq O(1)$.
  \item[(B)]  $\hat{F}(\{j\}) = \frac{1}{\sqrt{n}}$ for all $j \in [n]$.
  \item[(C)] 
   $\|\hat{F}\|_0 \leq 2^{O(\sqrt{n} \log(n))}$.
  \end{enumerate}
\end{theorem}

Our example follows the same outline as Bourgain's approach. 
Bourgain proved a stronger theorem
showing that there is a function satisfying (A) and (B) but
its Fourier sparsity in (C) is at most $2^{O(\sqrt{n})}$.
His construction starts by considering a simple function $H$ 
satisfying (A) and (B) but not (C).
He then carefully uses randomness to eliminate most
of the Fourier coefficients in $H$ while maintaining (A) and (B), 
and improving the sparsity.
We observe that it is enough to truncate $H$
to prove the theorem above.

Before proving the theorem, let us see how it yields
a limitation to Pisier's inequality.
Let $\mathcal{F} = \{S \subseteq [n] : \hat{F}(S) \neq \emptyset\}$, and
consider the function $f:\{\pm 1\}^{n} \rightarrow \R^{\mathcal{F}}$ 
defined by  
$$(f(x))_S = \hat{F}(S) \chi_S(x)$$
for each $S \in \mathcal{F}$.
Define a norm on $\R^{\mathcal{F}}$ as follows.
Every $v \in \R^{\mathcal{F}}$ corresponds to the function 
$g = g_v : \{\pm 1\}^{n} \to \R$ that is defined by $\hat{g}(S) = v_S$.
The norm of $v$ is defined to be
$$\|v\| = \|g\|_\infty = \max \{|g(z)| : z \in \{\pm 1\}^n\}.$$
It follows that for every $x \in \{ -1,1\}^n$,
$$\|f(x)\| = \|F\|_\infty \leq O(1)$$
and that
$$\|\lin{f}(x)\| \geq n \cdot \frac{1}{\sqrt{n}} \geq \Omega\Big( \frac{\log |\mathcal{F}|}{ \log \log |\mathcal{F}|}\Big).$$
It remains to prove the theorem.

\begin{proof}[Proof of Theorem~\ref{thm:Ex}]
First, we define a function $H : \{ -1,1\}^n \to \R$ by
  \[
  H(x) := \textrm{Im}\Big(  \prod_{j=1}^n \Big(1 + \frac{i}{\sqrt{n}} x_j \Big)\Big) = \sum_{S \subseteq [n]} \textrm{Im}\Big( \Big(\frac{i}{\sqrt{n}}\Big)^{|S|} \Big) \cdot \chi_S(x) ,
\]
where $\textrm{Im}$ denotes the imaginary part of a complex number. 
It follows that
$$\|H\|_\infty 
\leq \Big|1 + \frac{i}{\sqrt{n}}\Big|^n = \Big(\sqrt{1+\frac{1}{n}}\Big)^n \leq 3.$$
It also follows that
\begin{align}
\label{eqn:H1}
\hat{H}(\{j\}) = \frac{1}{\sqrt{n}}
\end{align}
for all $j \in [n]$ 
and 
$$|\hat{H}(S)| \leq n^{-|S|/2}$$ for all $S \subseteq [n]$. 

The function $F$ is obtained from $H$ by truncating the high frequencies.
Let  
\[
    F(x) := \sum_{S \in \mathcal{F}} \hat{H}(S) \cdot \chi_S(x),
  \]
  where $\mathcal{F} := \{ S \subseteq [n] : |S| \leq 3\sqrt{n} \}$.  
Property (A) of $F$ can be justified as follows.
  For every $x \in \{ -1,1\}^n$,
  \begin{eqnarray*}
    |H(x) - F(x)| &=& \Big|\sum_{S \subseteq [n]} (\hat{H}(S) - \hat{F}(S)) \cdot \chi_S(x)\Big|
                      \leq \sum_{S \not \in \mathcal{F}}  \underbrace{|\hat{H}(S)|}_{\leq n^{-|S|/2}} \cdot \underbrace{|\chi_S(x)|}_{\leq 1}                  \\  &\leq& \sum_{k>3\sqrt{n}} {n \choose k} n^{-k/2} \leq 
                                                                                                                                                                    \sum_{k>3\sqrt{n}}\Big(\frac{e\sqrt{n}}{k}\Big)^k 
                                                                                                                                                                  \leq 2^{-\Omega(\sqrt{n})} .
  \end{eqnarray*}
So, indeed
$\|F\|_\infty \leq \|H\|_\infty + \|H - F\|_\infty \leq O(1)$.
  Property (B) of $F$ holds by \eqref{eqn:H1}.
  Property (C) holds because $\|\hat{F}\|_0 \leq |\mathcal{F}| 
  \leq 2^{O(\log(n) \sqrt{n})}$.
\end{proof}
\begin{remark}
Bourgain used random sampling to sparsify the Fourier spectrum of $H$ and get sparsity $2^{O(\sqrt{n})}$.
Bourgain used Khinchine's inequality to analyze the sparsity
of the random function.
One can perform a similar analysis using more standard concentration
bounds. 
\end{remark}

\begin{remark}
Theorem~\ref{thm:Ex} can be proved with
$$F(x) = T_k(\tfrac{x_1+\ldots+x_n}{n})$$
as well, where $k = \lfloor \sqrt{n} \rfloor$
and $T_k$ is the $k$'th Chebyshev polynomial of the first kind.
\end{remark}

\section*{Acknowledgements}
We thank Mrigank Arora and Emanuel Milman for useful comments.

\bibliographystyle{abbrv}
\bibliography{pisier}
\end{document}